\documentclass[11pt,reqno,english,a4]{amsart}
\usepackage{amsmath,amsthm,amsfonts,amssymb,amscd}
\usepackage[latin1]{inputenc}
\usepackage{psfrag}
\usepackage{epsfig}
\usepackage{a4wide}
\usepackage[]{graphicx}
\usepackage{times}
\usepackage{bm}
\usepackage{mathdots}
\usepackage{mathrsfs}
\usepackage{caption}

\newcommand{\R}{\mathbb{R}}

\newcommand{\N}{\bf N}
\newcommand{\A}{\bf A}
\newcommand{\K}{\bf K}

\newcommand{\fg}{\mathfrak g}
\newcommand{\fa}{\mathfrak a}

\newcommand{\fk}{\mathfrak k}

\newcommand{\fs}{\mathfrak {so}}
\newcommand{\fn}{\mathfrak n}

\newcommand{\D}{\mathcal{D}}

\newtheorem{definition}{Definition}[section]

\newtheorem{theorem}[definition]{Theorem}

\newtheorem{cor}[definition]{Corollary}
\newtheorem{rem}[definition]{Remark}
\newtheorem{pro}[definition]{Proposition}

\def\co {cohomogeneity one }

\begin{document}
	
	\title[On cohomogeneity one linear actions on pseudo-Euclidean space $\mathbb{R}^{p,q}$]
	{On cohomogeneity one linear actions on pseudo-Euclidean space $\mathbb{R}^{p,q}$}
	
	\author{P. Ahmadi}	
	\author{S. Safari}

	\thanks{}
	
	
	
	\date{\today}
	\address{
		P. Ahmadi\\
		Departmental of mathematics\\
		University of Zanjan\\
		University blvd.\\
		Zanjan\\
		Iran}
	\email{p.ahmadi@znu.ac.ir}
	
	\address{
		S. Safari\\
		Departmental of mathematics\\
		University of Zanjan\\
		University blvd.\\
		Zanjan\\
		Iran}
	\email{salim.safari@znu.ac.ir}

	\begin{abstract}
		The aim of this paper is to study cohomogeneity one isometric linear actions on the $p+q$-dimensional pseudo-Euclidean space $\R^{p,q}$.  It is proved that the natural isometric action of the nilpotent factor of an Iwasawa decomposition of $SO(p,q)$ is not of cohomogeneity one. The orbits of cohomogeneity one actions of some subgroups of a maximal parabolic subgroup of the isometry group of $\R^{p,q}$ are determined and it is proved that there exist cohomogeneity one isometric actions on $\R^{p,q}$ which are orbit-equivalent on the complement of a $p$-dimensional degenerate subspace $\mathbb{W}^p$ of $\R^{p,q}$ and not orbit-equivalent on $\mathbb{W}^p$. 
	\end{abstract}
	
	\maketitle
	\medskip
	\medskip
	
	\thispagestyle{empty}
	

\section{Introduction and Preliminaries}

 The study of non-transitive actions of transformation groups on manifolds is an interesting problem. The first and most natural case is the case when the action has an orbit of codimension one, the so called \textit{cohomogeneity one} action. The concept of a cohomogeneity one
 action on a manifold $M$ was introduced by P.S. Mostert in his 1956 paper \cite{Mos}. The key hypothesis was the compactness of the acting Lie group in the paper. He assumed that the acting Lie group $G$ is compact and determined the orbit space up to homeomorphism. More precisely, he proved that by the cohomogeneity one action of a compact Lie group $G$ on a manifold $M$ the orbit space $M/G$ is homeomorphic to one of the spaces $\R$, $S^1$, $[0,1]$, or $[0,1)$. In the general case, in \cite{Ber} B. Bergery showed that if a Lie group acts on a manifold properly and with cohomogeneity one, then the orbit space $M/G$ is homeomorphic to one of the above spaces. 

A result by D. Alekseevsky in \cite{Alek} says that, for an arbitrary Lie group $G$,  the action of $G$ on $M$ is proper if and only if there is a complete $G$-invariant Riemannian metric $g$ on $M$ such that $G$ is a closed Lie subgroup of $Isom(M,g)$. This theorem provides a link between proper actions and Riemannian $G$-manifolds.

Cohomogeneity one Riemannian manifolds have been studied by many mathematicians (see, e.g., \cite{AA, Ber, Bernd, GWZ, Koll, MK, PS, S, V2}). The subject is still an active one. The common hypothesis in the theory is that the acting group is closed in the full isometry group of the Riemannian manifold and the action is isometrically. When the metric is indefinite, this assumption in general does not imply that the action is proper, so the study becomes much more complicated. Also, some of the results and  techniques of definite metrics fail for indefinite metrics (see, e.g., \cite{AK1, Pa, BDV}). 

Here we assume that $M$ is the pseudo-Euclidean space $\R^{p,q}$, that is the $p+q$-dimensional real
vector space $\R^{p+q}$ with the scalar product of signature $(p,q)$
given by  
\begin{equation}\label{scalar}\langle x,y\rangle=\sum_{i=1}^p x^iy^i - \sum_{j=1}^q x^jy^j,
\end{equation}
where $x=(x^1,\cdots,x^{p+q})$ and $y=(y^1,\cdots,y^{p+q})$, and $G$ is a closed Lie subgroup of the isometry group of $\R^{p,q}$ which acts on it with cohomogeneity one. Throughout the paper it is assumed that $p\geqslant q$, since $\R^{p,q}$ is anti-isometric to $\R^{q,p}$. 

The standard basis for $\R^{p,q}$ is denoted by $(e_1,..., e_{p+q})$. 
Let $\Pi_i$ be the hyperplane defined by the equation $x^{p-i+1}+x^{p+i}=0$, where $1\leqslant i\leqslant q$. Then $\Pi_i=w_i^\perp$ where $w_i=e_{p-i+1}-e_{p+i}$. To adjust the notations, we assume that $\Pi_0=\R^{p,q}$. If $p\neq q$, let $\mathscr{P}_j$ denote the hyperplane defined by $x^j=0$, where $1\leqslant j\leqslant p-q$.

 An isotropic subspace of $\R^{p,q}$ is a vector subspace $V\subset \R^{p,q}$ with the property that 
$$ \langle v,w\rangle=0;\quad \forall v,w\in V.$$
Let $V$ be a maximal isotropic subspace of $\R^{p,q}.$ Hence $\dim V=q$. The corresponding maximal parabolic subgroup of $SO_{o}(p,q)$ is the stabilizer of $V$ in $SO_{o}(p,q)$:
$$ Q=\{g\in SO_{o}(p,q)| g.V=V\}.$$
If $V=\bigcap_{j=1}^{p-q}\mathscr{P}_j\cap\bigcap_{i=1}^q\Pi_i$, then $Q=\K_{0}\A\N$, where the subgroups $\K_0$, $\A$ and $\N$ are introduced in section \ref{iwasawa}. One interesting class of \co actions on $\R^{p,q}$ is given by certain subgroups of the maximal parabolic subgroup $Q$. Our first result is Theorem \ref{N}. This result states that the action of $\N$ on $\R^{p,q}$ is not of cohomogeneity one, and its proof indicates that the action is of cohomogeneity two. We consider the actions of subgroups $\K'\A\N$ on $\R^{p,q}$, where $\K'\subseteq\K_0$, and investigate throughly the orbit structure of these actions (see Theorem \ref{mainresult}). Our results, Theorem \ref{mainresult} and Proposition \ref{degspace}, generalize Theorem 4.2 and Corollary 4.3 of \cite{BDV}. A notable feature of these actions is that there exists a $p$-dimensional degenerate subspace $\mathbb{W}^p$ of $\R^{p,q}$ such that the induced orbits of all of these actions on $\R^{p,q}\diagdown \mathbb{W}^p$ coincide, whereas the orbit structures are different on $\mathbb{W}^p$.

\section{Iwasawa decomposition of $ \mathfrak{so} (p,q) $}\label{Lie algebra}\label{iwasawa}
In this section we introduce a fixed Iwasawa decomposition of $\mathfrak{so}(p,q)$ which will be used in the sequel. We remind that $p\geqslant q$.

The Lie algebra $ \mathfrak{so}(p,q) $ of the linear group $ SO(p,q) $ is given by
\begin{align*}
\mathfrak{so}(p,q) = \left\lbrace  \left(\begin{array}{cc}A&B \\
B^{t} & D \end{array}\right): A\in \mathfrak{so}(p), D\in \mathfrak{so}(q) ~and~B\in {\rm M}_{p\times q}(\R) \right\}
\end{align*} 
where $M_{p\times q}(\R)$ denotes the space of  $p\times q$ real matrices.
The notation $X=\left(\begin{array}{cc}A&B \\
B^{t} & D \end{array}\right)$ is used for a typical element of $\mathfrak{so}(p,q)$ throughout the paper. Hence $ A=(A_{ij})\in \mathfrak{so}(p), D=(D_{ij})\in \mathfrak{so}(q)$ and $ B=(B_{ij})\in M_{p\times q}(\R)$, where $M_{p\times q}(\R)$ denotes the vector space of $p\times q$ real matrices. Obviously, $ A_{ij}=-A_{ji}$ and $ D_{ij}=-D_{ji}.$

The Cartan involution $\theta(X) =-X^{t}$ of $\fs (p,q)$ induces the Cartan decomposition 
$$\fs (p,q) = \fk\oplus \mathfrak{p} $$
with
\begin{align*}
\fk = \left\lbrace  \left(\begin{array}{cc}A&O \\
O & D \end{array}\right)\in \fs (p,q)\right\}\cong \fs (p)\times \fs (q),\quad \quad
\mathfrak{p}=\left\lbrace  \left(\begin{array}{cc}O&B \\
B^{t}& O \end{array}\right)\in \fs (p,q) \right\}\cong \R^{p\times q}.
\end{align*}

The subspace
\begin{align}\label{mathfrak{a}}
\mathfrak{a}=\left\lbrace  \left(\begin{array}{cc} O& C\\ C^{t} & O
 \end{array}\right)|\ C\in M_{p\times q}(\R) \right\}\subseteq \mathfrak{p}
\end{align}
that
\begin{align}
C=  \left(\begin{array}{ccccc} 0&&\cdots&&0\\\vdots &&& & \vdots \\  0 &&\cdots && 0 \\ 0 & &\cdots && c_{q}\\ & & &.&  \\&&.&&\\ & .&&&\\c_{1} && &&0
 \end{array}\right)
\end{align}
where $c_i\in \R$ for $1\leqslant i\leqslant q$, is a maximal abelian subspace of $\mathfrak{p}$. Let ${\bf A}=\exp (\mathfrak{a})$. The Lie subalgebra $\mathfrak{a}$ is abelian, so by a direct computation  one gets that
\begin{eqnarray*}\label{A}
	{\bf A}=\{X\in SO(p,q):X=\sum_{i=1}^{p-q} E_{ii}+\sum_{i=1}^{q} \{\cosh(c_{i})(E_{p-i+1,p-i+1}+E_{p+i,p+i})\\+\sinh(c_{i})(E_{p+i,p-i+1}+E_{p-i+1,p+i})\}\},
\end{eqnarray*}
where $ E_{ij} $ is the $(p+q)\times (p+q)$ matrix whose $ (i,j)$-entry is $ 1 $ and whose other entires are all $ 0. $

Let $ f_{i}$ be the member of $\mathfrak{a}^{*}$ whose value on the $\mathfrak{a}$ matrix indicated in $(\ref{mathfrak{a}})$ is $-c_{i}$. Then the restricted roots include all linear functional $\pm f_{i}\pm f_{j} $, with $i\neq j$. Also the $\pm f_{i} $ are restricted roots if $ p\neq q$. 
Then the restricted-root spaces for $ \pm f_{i}\pm f_{j} $ are $1$ -dimensional, and the restricted-root spaces for $\pm f_{i} $ have dimension $p-q.$
Let $\Sigma$ be the set of restricted roots and $\Sigma^{+} $ be the positive ones, and define $ \fn =\bigoplus_{\lambda\in \Sigma^{+}} g_{\lambda}.$ By Proposition 6.40-b of \cite[p. 370]{Kn}, $\mathfrak{n}$ is a Lie subalgebra of $\mathfrak{g}$ and is nilpotent. Then $ \mathfrak{so}(p,q) = \mathfrak{k}\oplus \mathfrak{a}\oplus \mathfrak{n} $ is an Iwasawa decomposition of $ \mathfrak{so}(p,q) $.
Let $\mathfrak{so}(p,q) = \fg_{0}\oplus\bigoplus_{\lambda\in \Sigma} \mathfrak{g}_{\lambda}$ be the restricted root space
decomposition of $\mathfrak{so}(p,q)$ induced by $\mathfrak{a}$.
 Explicitly, $\mathfrak{g}_{0} = \mathfrak{k}_{0}\oplus \mathfrak{a}$ (see \cite[p. 370]{Kn})
 that 
\begin{align*}
\mathfrak{k}_{0}= \left\lbrace \left(\begin{array}{cc} K&0\\0&0 \end{array}\right)\in SO(p,q): K\in \mathfrak{so}(p-q) \right\}\cong \mathfrak{so}(p-q)
\end{align*}
Denote by ${\bf K}_0$ the Lie subgroup $\exp (\mathfrak{k}_{0})$. 

In the following proposition we give an explicit form of each member of $\fn.$ For any element $X=\left(\begin{array}{cc}A&B \\
B^{t} & D \end{array}\right)$ of $ \mathfrak{so}(p,q),$ the entries of $A$ and $D,$ those are above the diagonal, determine the below ones. The following proposition shows that they characterize the entries of $ B,$ for any $X\in \fn,$ as well.
\begin{pro}\label{fn}
Let $X=\left(\begin{array}{cc}A&B \\
B^{t} & D \end{array}\right)$ and $ X\in \mathfrak{so}(p,q).$ Then $ X\in \fn$ if and only if
 \begin{eqnarray}\label{N-equation}
\begin{cases}
A_{k,p-l+1}=B_{k,l};\quad 1\leq k\leq p-q,\quad 1\leq l\leq q,\\ A_{p+1-j,p+1-i}=B_{p+1-j,i}; \quad\quad\ \  1\leq i<j\leq q,\\D_{ij}=-B_{p+1-i,j};\qquad\qquad\qquad 1\leq i<j\leq q.
\end{cases}
\end{eqnarray}
If $p=q$, then the equations introduced in the first line of $(\ref{N-equation})$ are omitted.
\end{pro}
\begin{proof}
First suppose that  $ p\neq q $. Let $H_{c_{1}\cdots c_{q}}$ denote a typical element of $\fa$ indicated in $(2).$ Define $ f_{i}:\mathfrak{a}\longrightarrow \R,$ by $ f_{i}(H_{c_{1}\cdots c_{q}})=-c_{i}.$ Then  by using a lexicographic ordering (see  \cite[p. 155]{Kn}), the set of restricted positive roots is
$$ \Sigma^{+}=\{f_{i}\pm f_{j}\}\cup \{f_{l}\},\quad 1\leq i<j\leq q,\quad 1\leq l\leq q.$$
If $\alpha \in \Sigma^+$ then the corresponding root space is defined by $$\fg_\alpha=\{X\in \fs (p,q) : ad(H)(X)=\alpha (H)X, \forall H\in \mathfrak{a}\}.$$
 Hence by a straightforward computation one gets that $\fg_{f_l}$ and $ \mathfrak{g}_{f_{i}\pm f_{j}}$ are constituted of all  
${\small \left( \begin{array}{cc}A & B\\ B^{t} & D \end{array}\right)}$ in $\mathfrak{so}(p,q)$, satisfying
\begin{eqnarray}\label{f_j}
 A_{k,p-l+1}=B_{k,l},\quad 1\leq k\leq p-q,\quad 1\leq l\leq q. 
\end{eqnarray}
 and
 \begin{eqnarray}\label{f_if_j}
 \begin{cases}
B_{p+1-i,j} c_{j}-B_{p+1-j,i} c_{i}=A_{p+1-j,p+1-i} (-c_{i}\pm c_{j})\\-A_{p+1-j,p+1-i} c_{i}-D_{i,j} c_{j}=B_{p+1-j,i} (-c_{i}\pm c_{j})\\A_{p+1-j,p+1-i} c_{j}+D_{i,j} c_{i}=B_{p+1-i,j} (-c_{i}\pm c_{j})\\B_{p+1-i,j} c_{i}-B_{p+1-j,i} c_{j}=D_{i,j} (-c_{i}\pm c_{j}),  \quad where\quad 1\leqslant i< j\leqslant q,
\end{cases}
\end{eqnarray}
respectively, and the other entries of $A$, $B$ and $D$ are zero. The equations (\ref{f_l}) gives the root-space $\fg_{f_j}$ explicitly. By using the equations (\ref{f_if_j}) one gets that $\fg_{f_i-f_j}$ and $\fg_{f_i+f_j}$ are one dimensional vector subspaces of $\mathfrak{so}(p,q)$ given by  $A_{p+1-j,p+1-i}= B_{p+1-j,i}=B_{p+1-i,j}=-D_{i,j}$ and  $ A_{p+1-j,p+1-i}=B_{p+1-j,i}=-B_{p+1-i,j}=D_{i,j}$, respectively (this means that their other entries are zero). Therefore, $\fg_{f_i-f_j}\oplus \fg_{f_i+f_j}$ is the two dimensional vector space given by 
$A_{p+1-j,p+1-i}=B_{p+1-j,i}$ and $D_{i,j}=-B_{p+1-i,j}$. These imply that 
$$\fn=\bigoplus_{1\leqslant l\leqslant q}\fg_{f_l}\oplus\bigoplus_{1\leqslant i<j\leqslant q}(\fg_{f_i-f_j}\oplus\fg_{f_i+f_j})$$
is of the claimed form.
\end{proof}
As a consequence of Proposition \ref{fn} one gets that $\dim {\bf N}=q(p-1)$.
\begin{rem}
	By proposition \ref{fn} every element of the nilpotent subalgebra $ \fn$ is in one of the following forms. (Here $ m(A_{ij})$ denotes the mirror image of $A_{ij}$ with respect to the dashed line between them, and the entries below the diagonal are determined by the upper ones, since $\fn\subseteq \mathfrak{s}(p,q)$).
	\end{rem}

\begin{figure}[!h]
	\centerline{\includegraphics[height=5.2cm]{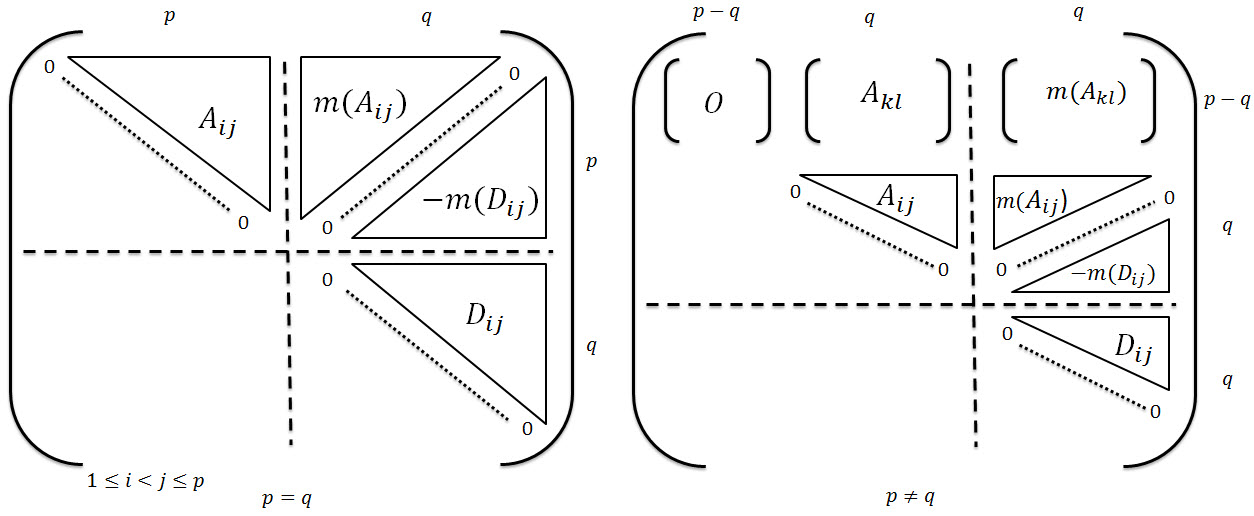}}
\end{figure}
	

\section{The action of the nilpotent factor ${\bf N}$ on $\R^{p,q}$ }
The nilpotent factor of $SO_\circ (p,q)$ is a $q(p-1)$ dimensional Lie subgroup. When $q=1$, then $\dim {\bf N}=p-1$ and so its action is not of \co on $\R^{p,1}$. In the following theorem we show that this result is true for arbitrary positive integer $q$.

\begin{theorem}\label{N}
The action of $\N$ on $\R^{p,q}$ is not of cohomogeneity one.
\end{theorem}

\begin{proof}
Let $x=\sum^{p+q}_{i=1}x^ie_i$ be an arbitrary fixed nonzero element of $\R^{p,q}$. We show that $\dim {\bf N} (x) \leqslant p+q-2$. 

Let $\fn_x$ be the Lie subalgebra of $\fn$ corresponding to the stabilizer subgroup at $x$, say ${\bf N}_x$. Hence
\begin{eqnarray*}
\fn_x &=&\{X\in \fn |\ \exp(tX)x=x,\ \forall t\in \R\} 
\\
 &=& \{X\in \fn |\ Xx=0\}.
\end{eqnarray*}
Let $X$ be a typical element of $\fn$. Then by Proposition \ref{fn}, the equation $Xx=0$ becomes as follows.

\begin{eqnarray}\label{equ-n}
\begin{cases}
\sum_{i=p-q+1}^{p}A_{ji}(x^{i}+x^{2p-i+1})=0 \ , \quad 1\leqslant j\leqslant p-q,\quad  {\rm if}\ p\neq q,\\

-\sum_{i=1}^{p-q}A_{i,p-q+1} x^{i} + \sum_{i=p-q+2}^{p} A_{p-q+1, i}(x^{i}+x^{2p-i+1})=0 \\
-\sum_{i=1}^{p-q+1}A_{i,p-q+2} x^{i} + \sum_{i=p-q+3}^{p}A_{p-q+2, i}(x^{i}+x^{2p-i+1})-D_{q-1,q}x^{p+q}=0\\
-\sum_{i=1}^{p-q+2}A_{i,p-q+3} x^{i} + \sum_{i=p-q+4}^{p}A_{p-q+3, i}(x^{i}+x^{2p-i+1})-\sum_{i=p-q+1}^{p-q+2}D_{q-2,p-i+1}x^{2p-i+1}=0\\
 \vdots \\
-\sum_{i=1}^{p-2}A_{i,p-1} x^{i} + A_{p-1, p}(x^{p}+x^{p+1})-\sum_{i=p-q+1}^{p-2}D_{2,p-i+1}x^{2p-i+1 }=0\\
\sum_{i=1}^{p-1}A_{i,p} x^{i} +\sum_{i=p-q+1}^{p-1}D_{1,p-i+1 }x^{2p-i+1 }=0\\
\sum_{i=1}^{p-2}A_{i,p-1} x^{i} - D_{1, 2}(x^{p}+x^{p+1})+\sum_{i=p-q+1}^{p-2}D_{2,p-i+1}x^{2p-i+1 }=0\\ \vdots \\
\sum_{i=1}^{p-q+2}A_{i,p-q+3} x^{i} - \sum_{i=p-q+4}^{p}D_{p-i+1, q-2}(x^{i}+x^{2p-i+1})+\sum_{i=p-q+1}^{p-q+2}D_{q-2,p-i+1}x^{2p-i+1}=0\\
\sum_{i=1}^{p-q+1}A_{i,p-q+2} x^{i} - \sum_{i=p-q+3}^{p}D_{p-i+1, q-1}(x^{i}+x^{2p-i+1})+D_{q-1,q}x^{p+q}=0\\ 
\sum_{i=1}^{p-q}A_{i,p-q+1} x^{i} - \sum_{i=p-q+2}^{p} D_{p-i+1, q}(x^{i}+x^{2p-i+1})=0 \\
\end{cases}
\end{eqnarray}
We look for the conditions on the point $x$ on which $\fn_x$ has minimum dimension.  One gets the result by considering the following three cases.

{\bf Case 1 :} $x\notin \bigcap_{i=1}^{q} \Pi_i$.

{\it Claim 1:} In this case $\dim {\bf N}(x)=p+q-(k+1)$, where $ k= min\{i: 1\leq i\leq q\ \ and\ \  x\notin \Pi_i\} $.

{\it Proof of Claim 1:}
For simplicity, first assume that $k=1$, i.e. $x^p+x^{p+1}\neq 0$. Then for $ 1\leq j\leq p-q $, if $p\neq q$, we have
$$ A_{jp}=-\frac{1}{x^{p}+x^{p+1}}\sum_{i=p-q+1}^{p-1} (x^{i}+x^{2p-i+1}) A_{ji}, $$
and for $ p-q+1\leq j\leq p-1 $ we have
$$ A_{jp}=\frac{1}{x^{p}+x^{p+1}}\lbrace -\sum_{i=j+1}^{p-1} (x^{i}+x^{2p-i+1}) A_{ji} + \sum_{i=1}^{j-1}\lbrace x^{i}A_{ij}+ x^{2p-i+1}D_{p-j+1,p-i+1}\rbrace \rbrace, $$
and
$$ D_{1,p-j+1}=\frac{1}{x^{p}+x^{p+1}}\lbrace -\sum_{i=j+1}^{p-1} (x^{i}+x^{2p-i+1}) D_{p-i+1,p-j+1} + \sum_{i=1}^{j-1}\lbrace x^{i}A_{ij}+x^{2p-i+1}D_{p-j+1,p-i+1}\rbrace \rbrace ,$$ 
where $x_k=0$ for any $k>p+q$
(note that $D_{ij}=0$ for $i,j\geq q+1$). 
 These imply that $ A_{jp}$ and $D_{1,p-i+1} $ are linear functions of the other entries, where $ 1\leq j\leq p-1$ and $p-q+1\leqslant i\leqslant p-1$. Therefore $\dim \fn_{x}=q(p-1)-(p-1)-(q-1)=(p-2)(q-1),$ which implies that $\dim {\bf N}(x)=\dim \fn - \dim\fn_x=p+q-2$. 
 
 Now let $1<k< q$. This implies that $x^{p-k+1}+x^{p+k}\neq 0$ and $x^{p-i+1}+x^{p+i}=0$ for any $i$ where $1\leqslant i< k$. Then for $1\leq j\leq p-q$, if $p\neq q$, we have
 $$ A_{j,p-k+1}=-\frac{1}{x^{p-k+1}+x^{p+k}}\sum_{i=p-q+1,i\neq p-k+1}^{p} (x^{i}+x^{2p-i+1}) A_{ji}, $$
 for $p-q+1\leq j\leq p-k$ we have
 $$ A_{j,p-k+1}=\frac{1}{x^{p-k+1}+x^{p+k}}\lbrace -\sum_{\stackrel{i=j+1}{ i\neq p-k+1}}^{p} (x^{i}+x^{2p-i+1}) A_{ji} + \sum_{i=1}^{j-1}\lbrace x^{i}A_{ij}+x^{2p-i+1}D_{p-j+1,p-i+1}\rbrace \rbrace, $$
 and
 $$ D_{k,p-j+1}=\frac{1}{x^{p-k+1}+x^{p+k}}\lbrace-\sum_{\stackrel{i=j+1}{ i\neq p-k+1}}^{p}(x^{i}+x^{2p-i+1}) D_{p-i+1,p-j+1} + \sum_{i=1}^{j-1}\lbrace x^{i}A_{ij}+x^{2p-i+1}D_{p-j+1,p-i+1}\rbrace \rbrace ,$$ 
 and for $p-k+2\leq j\leq p$, if $x^{p-k+1}\neq 0$ we have either
 $$ A_{p-k+1,j}=-\frac{1}{x^{p-k+1}} \sum_{i=1, i\neq p-k+1}^{j-1}\lbrace x^{i}A_{ij}+x^{2p-i+1}D_{p-j+1,p-i+1}\rbrace, $$
 or
  $$ D_{p-j+1,k}=-\frac{1}{x^{p+k}} \sum_{i=1,i\neq p-k+1}^{j-1}\lbrace x^{i}A_{ij}+x^{2p-i+1}D_{p-j+1,p-i+1}\rbrace .$$
  These imply that $A_{j,p-k+1}$ and $ D_{k,p-j+l},$ where $1\leq j\leq p-k$ and $p-q+1\leq j\leq p-k$ and one of $A_{p-k+1,j}$ or $D_{p-j+1,k},$ where $p-k+2\leq j\leq p$ are linear functions of the other entries. Therefore $\dim \fn_{x}=q(p-1)-(p+q-(k+1)),$ which implies that $\dim {\bf N}(x)=\dim \fn - \dim\fn_x=p+q-(k+1)$.
  
  Finally, let $k=q$. Then for $1\leq j\leq p-q$, if $p\neq q,$ we have
  $$ A_{j,p-q+1}=-\frac{1}{x^{p-q+1}+x^{p+q}}\sum_{i=p-q+2}^{p} (x^{i}+x^{2p-i+1}) A_{ji}, $$
  
  and for $p-q+2\leq j\leq p$, if $x^{p-q+1}\neq 0$ we have either
  $$ A_{p-q+1,j}=-\frac{1}{x^{p-q+1}} \sum_{i=1, i\neq p-q+1}^{j-1}\lbrace x^{i}A_{ij}+x^{2p-i+1}D_{p-j+1,p-i+1}\rbrace \rbrace, $$
  or
  $$ D_{p-j+1,q}=-\frac{1}{x^{p+q}} \sum_{i=1,i\neq p-q+1}^{j-1}\lbrace x^{i}A_{ij}+x^{2p-i+1}D_{p-j+1,p-i+1}\rbrace ,$$
  These imply that $A_{j,p-q+1}$ where $1\leq j\leq p-q$ and one of the $A_{p-q+1,j}$ or $D_{p-j+1,q},$ where $p-q+2\leq j\leq p$ are linear functions of the other entries. Therefore $\dim \fn_{x}=q(p-1)-(p-1),$ which implies that $\dim {\bf N}(x)=\dim \fn - \dim\fn_x=p-1$.

{\bf Case 2:} $x\in \bigcap_{i=1}^q\Pi_i\diagdown \bigcap_{j=1}^{p-q}\mathscr{P}_j$.
 \\
{\it Claim 2:} In this case $\dim {\bf N}(x)=q$.

{\it Proof of Claim 2:} The condition on $x$ implies that $x^{p-i+1}+x^{p+i}=0$ for all $1\leq i\leq q$ and $x^{j}\neq 0 $  for some $1\leq j\leq p-q.$ Then the system of equations (\ref{equ-n}) reduces to the following system.

\begin{eqnarray}\label{case2-n}
\begin{cases}
\sum_{i=1}^{p-1}A_{i,p} x^{i} +\sum_{i=p-q+1}^{p-1}D_{1,p-i+1 }x^{2p-i+1 }=0\\
\sum_{i=1}^{p-2}A_{i,p-1} x^{i})+\sum_{i=p-q+1}^{p-2}D_{2,p-i+1}x^{2p-i+1 }=0\\ \vdots \\
\sum_{i=1}^{p-q+2}A_{i,p-q+3} x^{i} +\sum_{p-q+1}^{p-q+2}D_{q-2,p-i+1}x^{2p-i+1}=0\\
\sum_{i=1}^{p-q+1}A_{i,p-q+2} x^{i} +D_{q-1,q}x^{p+q}=0\\ 
\sum_{i=1}^{p-q}A_{i,p-q+1} x^{i} =0. \\
\end{cases}
\end{eqnarray}
Hence for $p-q+1\leq m\leq p$ we have 
$$ A_{jm}=-\frac{1}{x^{j}}\lbrace \sum_{i=1, i\neq j}^{m-1} x^{i} A_{im} + \sum_{i=p-q+1}^{m-1} x^{2p-i+1}D_{p-m-1,p-i+1}\rbrace, $$
These imply that $ A_{jm}$ are linear functions of the other entries, where $ p-q+1\leq m\leq p.$ Therefore $\dim \fn_{x}=q(p-1)-q=q(p-2),$ which implies that $\dim {\bf N}(x)=\dim \fn - \dim\fn_x=q$.

{\bf Case 3 :} $x\in \bigcap_{i=1}^q\Pi_i\cap \bigcap_{j=1}^{p-q}\mathscr{P}_j$.\\
Hence, $ x^{i}=0$ for all $1\leq i\leq p-q,$ $x^{p-i+1}+x^{p+i}=0 $ for all $1\leq i\leq q$ and consequently $x^{p+j}=-x^{p-j+1}\neq 0 $  for some $ 1\leq j\leq q $. \\
{\it Claim 3:} In this case $\dim {\bf N}(x)=l-1$, where $ l=max\{j: 1\leq j\leq q \quad\textit{and}\quad x^{p+j}=-x^{p-j+1}\neq 0 \}.$

{\it Proof of Claim 3:} Let $x^{p-l+1}=-x^{p+l}\neq 0$ where $ l=max\{j: 1\leq j\leq q \quad\textit{and}\quad x^{p+j}=-x^{p-j+1}\neq 0 \}.$ If $l=1$, then $x=r(e_p-e_{p+1})$ for some $r\in \R$, which implies that ${\bf N}(x)=\{x\}$ and so $\dim {\bf N}(x)=0$. If $l>1$, then the system of equations (\ref{equ-n}) reduces to the following system.
\begin{eqnarray}\label{case3-n}
\begin{cases}
\sum_{i=p-l+1}^{p-1}\lbrace A_{i,p} x^{i} + D_{1,p-i+1 }x^{2p-i+1 }\rbrace=0\\
\vdots \\
\sum_{i=p-l+1}^{p-l+1}\lbrace A_{i,p-l+2} x^{i} + D_{l-1,p-i+1}x^{2p-i+1}\rbrace=0.
\end{cases}
\end{eqnarray}
Hence for $1\leq m\leq l-1$ we have
$$ D_{ml}=-\frac{1}{x^{p+l}}\lbrace \sum_{i=p-q+1}^{p-m} x^{i} A_{i,p-m+1} + \sum_{i=m+1,i\neq l}^{q} x^{p+i}D_{mi}\rbrace.$$
These imply that 
$ D_{ml}$ are linear functions of the other entries, where $ 1\leq m\leq l-1.$ Therefore $\dim \fn_{x}=q(p-1)-(l-1),$ and so $\dim {\bf N}(x)=\dim \fn - \dim\fn_x=l-1$.
\end{proof}

\begin{rem}\label{rem1}
In the following ordered basis of $\R^{p+q}$ the matrix representation of any element of $ \mathfrak{n} $ is a strictly upper triangular matrix. This implies that each element of $\fn$ is a nilpotent matrix.
\begin{equation}\label{basis}
( e_{p}-e_{p+1}, e_{p-1}-e_{p+2}, \cdots , e_{p-q+1}-e_{p+q}, e_{1}, \cdots, e_{q}, e_{p-q+1}+e_{p+q},\cdots, e_{p}+e_{p+1} ). 
\end{equation}
For the case $ p=q $, the vectors $e_{i}$, where $1\leq i\leq q$, are discarded from the basis (we remind that $w_i=e_{p-i+1}-e_{p+i}$, where $1\leqslant i\leqslant q$).

The representation of any element $X\in \mathfrak{n} $ in the basis (\ref{basis}) for $\R^{p,q}$ is determined as follows.\\
\begin{eqnarray}
\begin{cases}
X(e_p-e_{p+1})=0,\\
X(e_{p-j+1}-e_{p+j})=\sum_{i=1}^{j-1}(D_{i,j}-A_{p-j+1,p-i+1})(e_{p-i+1}-e_{p+i});\quad 2\leq j\leq q\\
X(e_{j})=-\sum_{i=1}^{q} A_{j,p-i+1}(e_{p-i+1}-e_{p+i});\quad 1\leq j\leq q\\
X(e_{p-k}+e_{p+k+1})=-\sum_{i=1}^{k}(D_{i,k+1}+A_{p-k,p-i+1})(e_{p-i+1}-e_{p+i})+2\sum_{i=1}^{p-q}A_{i,p-k} e_{i} \ , \quad k= q-j\\
+ \sum_{i=1}^{j-1}\left\lbrace (A_{p-l,p-k}+D_{k+1,l+1})(e_{p-l}-e_{p+l+1})+(A_{p-l,p-k}
+D_{k+1,l+1})(e_{p-l}+e_{p+l+1})\right\rbrace \ , \quad l= q-i
\end{cases}
\end{eqnarray}
These imply that ${\bf N}(w_1)=\{w_1\}$ and ${\bf N}(w_j)=\Sigma_{i=1}^{j-1}\R w_i+w_j$, where $2\leqslant j\leqslant q$. Also ${\bf N}(e_i)=e_i+\Sigma_{i=1}^q\R w_i$.  
And the representation of any element $Y\in \mathfrak{a} $ is determined as follows.
\begin{eqnarray}
\begin{cases}
Y(e_{p-(i-1)}-e_{p+i} )=-c_{i} (e_{p-(i-1) }-e_{p+i} )\\
Y(e_{i })=0\\
Y(e_{p-(q-i) }+e_{p+(q-i)+1} )=c_{(q-i)+1} (e_{p-(q-i) }+e_{p+(q-i)+1} ),
\end{cases}
\end{eqnarray}
where $1\leq i\leq q$. Hence for any $i$ where $1\leqslant i\leqslant q$, we have ${\bf A}(w_i)=\R_+w_i$, ${\bf A}(e_i)=\{e_i\}$ and ${\bf K}_0(w_i)=\{w_i\}$. Thus ${\bf K}_0{\bf A}{\bf N}(\cap_{i=1}^j\Pi_i)\subseteq \cap_{i=1}^j\Pi_i$, where $1\leqslant j\leqslant q$. This fact is a key point in determining of the orbits of the group $Q={\bf K}_0\ \A\N$ in the hyperquadrics in $\R^{p,q}$ (see the proofs of Corollaries \ref{co-Lambda} and \ref{co-H^{p,q-1}}).

\end{rem}

\begin{pro}
      The direction $\R(e_{p}-e_{p+1})$ is the only direction in $\R^{p,q}$ which is preserved by $\N$. This direction is fixed point-wise by $\N$.
\end{pro}
\begin{proof}
 Let $\R x$ be a direction which is preserved by $\N$. So for any $X\in \fn$ there is a differentiable function $\lambda_X:\R\rightarrow \R$ such that $$\exp (tX)x=\lambda_X(t)x\ ,\quad \forall t\in \R.$$
This implies that $Xx=\frac{d}{dt}|_{t=0}\lambda_X(t)x$. Hence $\frac{d}{dt}|_{t=0}\lambda_X(t)$ is an eigenvalue of $X$ and so $\frac{d}{dt}|_{t=0}\lambda_X(t)=0$ by Remark \ref{rem1}. Thus $$Xx=0,\quad \forall X\in\fn.$$
By Remark \ref{rem1} it is obvious that $x\in \R(e_p-e_{p+1})$.
\end{proof}
\section{The action of $SO_\circ(p,q)$ on $\R^{p,q}$}

In this section we study the orbits of the action of the Lie subgroup $SO_\circ(p,q)$ of $Iso(\R^{p,q} )$ on $\R^{p,q}$. We first introduce some notations:
\begin{eqnarray}\label{notations}
\mathcal{S}^{p,q}=\{v\in \R^{p,q}:\langle v,v\rangle >0\},\quad \quad \mathcal{T}^{p,q}=\{v\in \R^{p,q}:\langle v,v\rangle <0\},\\
\Lambda^{p+q-1} = \{v\in \R^{p,q}-\{0\}:\langle v,v\rangle =0\},\\
\Lambda^{p}_{+}=\{v\in \Lambda^p: \langle v,e_{p+1}\rangle <0 \},\\
\Lambda^{p}_{-}=\{v\in \Lambda^p: \langle v,e_{p+1}\rangle >0\}
\end{eqnarray}
$ \mathcal{S}^{p,q}, \mathcal{T}^{p,q} $ and $ \Lambda^{p+q-1} $ are the set of space-like, time-like and light-like vectors in $ \R^{p,q},$ respectively.
By proposition 4.22 of \cite{ON}, the nullcone $ \Lambda^{p+q-1} $ of $ \R^{p,q} $ is a hypersurface invariant under scalar multiplication and diffeomorphic to $ (\R^{q}-\{0\})\times \mathbb{S}^{p-1}$, where $\mathbb{S}^{p-1}$ is the $p-1$ dimensional Euclidean sphere with redius $1$ in $\R^{p}$. Hence $\Lambda^{p+q-1}$ is connected if and only if $q>1$. For $q=1$, the index $\pm$ in $(13)$ and $(14)$ refers to time-orientation.

For $r\in \R_{+}$ we define
\begin{eqnarray}
\mathbb{S}^{p-1,q}(r)=\{v\in \mathcal{S}^{p,q}:\langle v,v\rangle =r^{2}\}\quad, \quad \mathbb{H}^{p,q-1}(r)=\{v\in \mathcal{T}^{p,q}:\langle v,v\rangle =-r^{2}\}.
\end{eqnarray}

The pseudo-Riemannian submanifolds $\mathbb{S}^{p-1,q}(r)$ and  $\mathbb{H}^{p,q-1}(r)$ are called $p+q-1$ dimensional pseudo-sphere and pseudo-hyperbolic space of radius $r$, respectively (see \cite[Ch.4]{ON}). The pseudo-sphere $\mathbb{S}^{p-1,q}(r)$ is diffeomorphic to $\mathbb{S}^{p-1}\times\R^q $ and the pseudo-hyperbolic space $\mathbb{H}^{p,q-1}(r)$ is diffeomorphic to $\R^p\times \mathbb{S}^{q-1}.$ Hence  $\mathbb{H}^{p,q-1}(r)$ is connected if and only if $q>1$. If $q=1$, then $\mathbb{H}^{p,0}(r)=\mathbb{H}^{p,0}_+(r)\cup \mathbb{H}^{p,0}_-(r)$, where 
$$\mathbb{H}^{p,0}_+(r)=\{v\in \R^{p,1}\ :\langle v,v\rangle=-r^2 \ {\rm and}\ \langle v,e_{p+1}\rangle >0\},$$ and
$$\mathbb{H}^{p,0}_-(r)=\{v\in \R^{p,1}\ :\langle v,v\rangle=-r^2 \ {\rm and}\ \langle v,e_{p+1}\rangle <0\}.$$
The induced metric on $\mathbb{H}^{p,0}_+(r)$ is Riemannian and $\mathbb{H}^{p,0}_+(r)$ is the well-known hyperboloid model of $p$ dimensional real hyperbolic space with constant curvature $-r^{-2}$. We have $Iso(\mathbb{H}^{p,0}_+(r))=SO(p,1)$. In particular, $\mathbb{H}^{p,0}_+(r)$ is an orbit of $SO_\circ(p,1)$ and the isotropy subgroup at a point of $\mathbb{H}^{p,0}_+(r)$ is isomorphic to ${\bf K}=SO(p)$. Hence, as a homogeneous space, $\mathbb{H}^{p,0}_+(r)= SO_\circ(p,1)/ SO(p)=SO_\circ(p,1)/ \K$. Using the time-reversing map $\R^{p,1}\rightarrow \R^{p,1}$, $(x^1,...,x^p, x^{p+1})\mapsto (x^1,...,x^p , -x^{p+1})$, one gets that $\mathbb{H}^{p,0}_-(r)$ is another orbit of $SO_\circ(p,1)$ and therefore $\mathbb{H}^{p,0}_-(r)=SO_\circ(p,1)/SO(p)=SO_\circ(p,1)/\K$.

 Let $q>1$. Then $Iso(\mathbb{S}^{p-1,q}(r))=Iso(\mathbb{H}^{p,q-1}(r))=O(p,q)$ (see \cite[p.239]{ON}). In particular, $\mathbb{S}^{p-1,q}(r)$ and $\mathbb{H}^{p,q-1}(r)$ are orbits of $SO_\circ(p,q)$. The isotropy subgroup at a point of $\mathbb{S}^{p-1,q}(r)$ is isomorphic to $SO_\circ(p-1,q)$ and that of  $\mathbb{H}^{p,q-1}(r)$ is isomorphic to $SO_\circ(p,q-1)$. Thus, as homogeneous spaces, $\mathbb{S}^{p-1,q}(r)=SO_\circ (p,q)/ SO_\circ (p-1,q)$ and $\mathbb{H}^{p,q-1}(r)=SO_\circ (p,q)/ SO_\circ (p,q-1)$.  

Finally, $SO_\circ (p,q)$ leaves $\Lambda^{p+q-1}$ invariant. If $q=1$, then
$\Lambda^{p}_{+}$ and $\Lambda^{p}_{-}$ are the orbits. The isotropy group of $SO_\circ(p,1)$ at a point in $\Lambda^{p}_+$ or $\Lambda^{p}_{-}$ is isomorphic to the subgroup ${\bf K}_0\N$ of $SO_\circ(p,1)$. Thus, as homogeneous spaces, we have $\Lambda^{p}_{-}=SO_\circ(p,1)/ {\bf K}_0\N$ and $\Lambda^{p}_{+}=SO_\circ(p,1)/{\bf K}_0\N$. In the case that $q>1$, we have $\Lambda^{p+q-1}=SO_\circ (p,q)/ {\bf K}_0{\bf N}$.

Altogether it follows that we have the following decomposition $\mathcal{F}_{SO_\circ (p,q)}$ of $\R^{p,q}$ into orbits of $SO_\circ(p,q)$.

$$\mathcal{F}_{SO_\circ(p,1)}=\{0\}\cup \Lambda^{p}_{\pm}\cup \bigcup_{r\in \R_+}\mathbb{H}^{p,0}_{\pm}(r)\cup \bigcup_{r\in \R_+}\mathbb{S}^{p-1,1}(r),$$
and
$$\mathcal{F}_{SO_\circ(p,q)}=\{0\}\cup \Lambda^{p+q-1}\cup \bigcup_{r\in \R_+}\mathbb{H}^{p,q-1}(r)\cup \bigcup_{r\in \R_+}\mathbb{S}^{p-1,q}(r) \ ,\ {\rm where}\ q>1.$$

\section{Cohomogeneity one actions of Lie subgroups of ${\bf K}_{\circ}{\bf AN}$ on $\R^{p,q}$}

Let $G={\bf K}'\A\N$ and ${\bf K}' \subseteq {\bf K}_0$. Consider the natural action of $G$ on $\R^{p,q}$. For the case $q=1$, in \cite[Sec.4]{BDV} it is claimed that $G$ has exactly two orbits on the positive light cone $\Lambda^{p}_+$, namely $\R_-w_1$ and $\Lambda^{p}_+- \R_-w_1$.  The argument is analogous for the negative light cone $\Lambda^{p}_-$.   In other words, $G$ has two orbits with dimension $p-1$ and two orbits with dimension one. Also $G$ acts transitively on each connected component of $\mathbb{H}^{p,0}$, i.e. $G$ has two orbits with dimension $p$ in $\mathbb{H}^{p,0}$. Here, we generalize these consequences for general positive $q$ (see Corollaries \ref{co-Lambda} and \ref{co-H^{p,q-1}}).

The special Euclidean group ${\bf  K}_{0}{\bf N} $ fixes the vector $ w_{1}=e_{p}-e_{p+1}\in \Lambda^{p+q-1}$ and for Lie subgroup $ Q={\bf K}_{0}{\bf A}{\bf N} $ of $ SO(p,q) $ the orbit $ Q(w_{1}) $ is equal to  $Q(w_{1})={\bf A}(w_{1})=\R_{+}w_{1}=\R w_{1}\cap \Lambda^{p+q-1}$. Similarly, we have $ Q(-w_{1})={\bf A}(-w_{1})=\R_{-}w_{1}=\R w_{1}\cap \Lambda^{p+q-1}.$

\begin{pro}\label{KAN}
	
	Let $G={\bf K}_{0}{\bf AN}$, which acts on $ \R^{p,q} $ naturally. Let $x=\sum_{i=1}^{p+q} x^{i}e_{i} $ be a nonzero element of $\R^{p,q}$ and $\dim G(x)=m$. Then one of the following statements hold.
	
	a) If $x\notin \bigcap_{i=1}^q\Pi_i$ then $m=p+q-k$, where $ k= min\{i: 1\leq i\leq q\ \ and\ \  x^{p-i+1}+x^{p+i}\neq 0\} $. 
	
	b) If $x\in \bigcap_{i=1}^q\Pi_i\diagdown \bigcap_{j=1}^{p-q}\mathscr{P}_j$ then $m=p-1$.
	
	c) If $x\in \bigcap_{i=1}^q\Pi_i\cap \bigcap_{j=1}^{p-q}\mathscr{P}_j$, then $m=max\{j: 1\leq j\leq q \quad\textit{and}\quad x^{p+j}=-x^{p-j+1}\neq 0 \}.$
\end{pro}
\begin{proof}
	If $p=q,$ then $\mathfrak{k}_{0}=0$ and the action reduces to the action of ${\bf A}{\bf N}$. For the case $p\neq q,$ let $K_{ij}$ denote a typical entry of $\fk_{0},$ where $1\leq i,j\leq p-q.$ Then $K_{ij}=-K_{ji}.$ Let $ x=\sum_{i=1}^{p+q} x^{i}e_{i} $ be an arbitrary fixed element of $ \R^{p,q} $. Let $ (\mathfrak{k}_{0}\oplus\fa\oplus \fn)_{x} $ be the Lie subalgebra of $ \mathfrak{k}_{0}\oplus\fa\oplus \fn $ corresponding to the stabilizer subgroup at $ x $, say $ ({\bf K}_{0}{\bf A}{\bf N})_{x} $. Hence 
	\begin{eqnarray*}
		(\mathfrak{k}_{0}\oplus\fa\oplus \fn)_{x}&=&\{X\in \mathfrak{k}_{0}\oplus\fa\oplus \fn | exp(tX)x=x, \quad \forall t\in \R\}
		\\
		&=& \{X\in \mathfrak{k}_{0}\oplus\fa\oplus \fn |\ Xx=0\}.
	\end{eqnarray*}
	Let $X$ be a typical element of $\mathfrak{k}_{0}\oplus\fa\oplus \fn$, then by (1) and Proposition \ref{fn} one gets that $X$ is in one of the following forms.
	\begin{figure}[!h]\label{figure2}
		\centerline{\includegraphics[height=5cm]{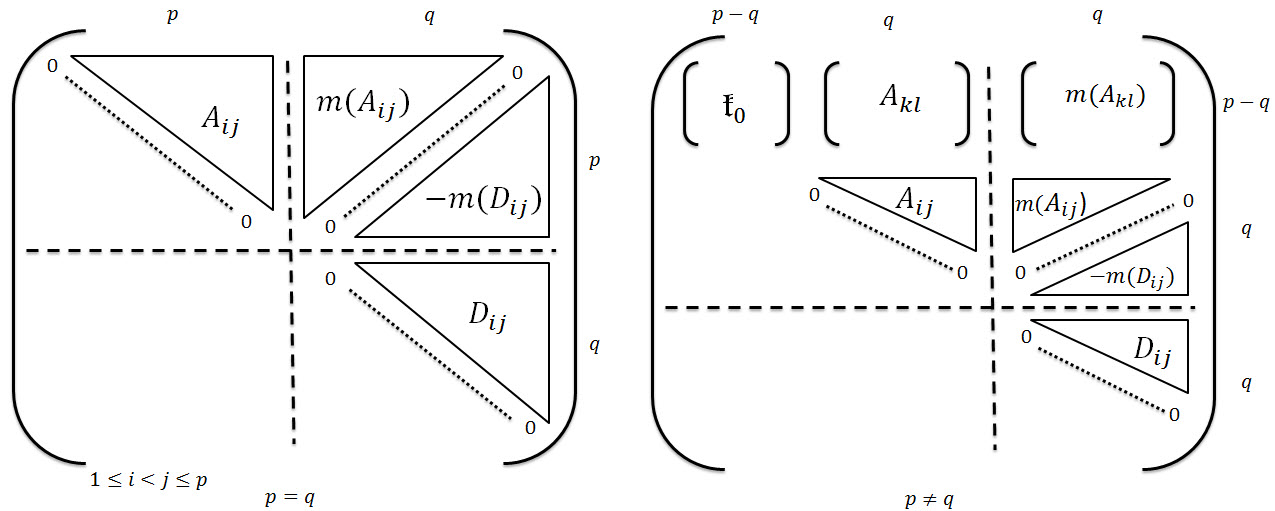}}
	\end{figure}
	
	Hence the equation $ Xx=0 $ becomes as follows.
	\begin{eqnarray}\label{KAN-equ}
		\begin{cases}
			\sum_{i=j+1}^{p-q}K_{ji}x^{i}-\sum_{i=1}^{j-1}K_{ij}x^{i}+\sum_{i=p-q+1}^{p}A_{ji}(x^{i}+x^{2p-i+1})=0 \ , \quad 1\leqslant j\leqslant p-q,\quad  {\rm if}\ p\neq q,\\
			-\sum_{i=1}^{p-q}A_{i,p-q+1} x^{i} + \sum_{i=p-q+2}^{p} A_{p-q+1, i}(x^{i}+x^{2p-i+1}) + c_{q}x^{p+q}=0 \\
			-\sum_{i=1}^{p-q+1}A_{i,p-q+2} x^{i} + \sum_{i=p-q+3}^{p}A_{p-q+2, i}(x^{i}+x^{2p-i+1})+c_{q-1}x^{p+q-1}-D_{q-1,q}x^{p+q}=0\\ 
			-\sum_{i=1}^{p-q+2}A_{i,p-q+3} x^{i} + \sum_{i=p-q+4}^{p}A_{p-q+3, i}(x^{i}+x^{2p-i+1})+c_{q-2}x^{p+q-2}\\\quad -\sum_{i=p-q+1}^{p-q+2}D_{q-2,p-i+1}x^{2p-i+1}=0\\
			\vdots \\
			-\sum_{i=1}^{p-2}A_{i,p-1}x^{i}+A_{p-1,p}(x^{p}+x^{p+1})+c_{2}x^{p+2}
			-\sum_{i=p-q+1}^{p-2}D_{2,p-i+1}x^{2p-i+1 }=0\\
			-\sum_{i=1}^{p-1}A_{i,p} x^{i} +c_{1}x^{p+1}-\sum_{i=p-q+1}^{p-1}D_{1,p-i+1 }x^{2p-i+1 }=0\\
			\sum_{i=1}^{p-1}A_{i,p} x^{i} +c_{1}x^{p}+\sum_{i=p-q+1}^{p-1}D_{1,p-i+1 }x^{2p-i+1 }=0\\
			\sum_{i=1}^{p-2}A_{i,p-1} x^{i} +c_{2}x^{p-1}- D_{1, 2}(x^{p}+x^{p+1})+\sum_{i=p-q+1}^{p-2}D_{2,p-i+1}x^{2p-i+1 }=0\\
			\vdots\\
			\sum_{i=1}^{p-q+2}A_{i,p-q+3} x^{i} +c_{q-2}x^{p-q+3}- \sum_{i=p-q+4}^{p}D_{p-i+1, q-2}(x^{i}+x^{2p-i+1})\\\quad +\sum_{i=p-q+1}^{p-q+2}D_{q-2,p-i+1}x^{2p-i+1}=0\\
			\sum_{i=1}^{p-q+1}A_{i,p-q+2} x^{i} +c_{q-1}x^{p-q+2}- \sum_{i=p-q+3}^{p}D_{p-i+1, q-1}(x^{i}+x^{2p-i+1})+D_{q-1,q}x^{p+q}=0\\
			\sum_{i=1}^{p-q}A_{i,p-q+1} x^{i} +c_{q}x^{p-q+1}-\sum_{i=p-q+2}^{p} D_{p-i+1, q}(x^{i}+x^{2p-i+1})=0 \\
		\end{cases}
	\end{eqnarray}

	The result is an immediate consequence of the following three cases for $x=\sum_{i=1}^{p+q} x^{i}e_{i}$, which are indicated in Proposition \ref{KAN} (a), (b) and (c), in the  system of equations (\ref{KAN-equ}) and using the fact that $\dim G(x)=\dim G-\dim G_{x}$.
	
	{\bf Case 1 :} 	$x\notin \bigcap_{i=1}^q\Pi_i$.
	
	{\it Claim 1:} In this case $\dim {\bf K}_{0}{\bf A}{\bf N}(x)=p+q-k$, where $ k= min\{i: 1\leq i\leq q\ \ and\ \  x\notin \Pi_i\} $.

	{\it Proof of Calim 1:}
	For simplicity, first assume that $k=1$. Then for $ 1\leq j\leq p-q $, if $p\neq q$, we have
	$$ A_{jp}=\frac{1}{x^{p}+x^{p+1}}\lbrace -\sum_{i=p-q+1}^{p-1} (x^{i}+x^{2p-i+1}) A_{ji}+\sum_{i=1}^{j-1} x^{i}K_{ij}-\sum_{i=j+1}^{p-q} x^{i}K_{ji}\rbrace, $$
	and for $ p-q+1\leq j\leq p-1 $ we have
	$$ A_{jp}=\frac{1}{x^{p}+x^{p+1}}\lbrace -\sum_{i=j+1}^{p-1} (x^{i}+x^{2p-i+1}) A_{ji} + \sum_{i=1}^{j-1}\lbrace x^{i}A_{ij}+x^{2p-i+1}D_{p-j+1,p-i+1}\rbrace-x^{2p-j+1}c_{p-j+1}\rbrace, $$
	and
	$$ D_{1,p-j+1}=\frac{1}{x^{p}+x^{p+1}}\lbrace -\sum_{i=j+1}^{p-1} (x^{i}+x^{2p-i+1}) D_{p-i+1,p-j+1} + \sum_{i=1}^{j-1}\lbrace x^{i}A_{ij}+x^{2p-i+1}D_{p-j+1,p-i+1}\rbrace+x^{j}c_{p-j+1}\rbrace ,$$ 
	where $x_k=0$ for any $k>p+q$
	(note that $D_{ij}=0$ for $i,j\geq q+1$).
	Summing the sides of the two equations including $c_{1},$ indicated in the lines six and seven in (\ref{KAN-equ}), and using the fact that $x^{p}+x^{p+1}\neq0$ one gets that $c_{1}=0.$ 
	These imply that $ A_{jp}$ and $D_{1,p-i+1} $ are linear functions of the other entries, where $ 1\leq j\leq p-1$ and $p-q+1\leqslant i\leqslant p-1$ and $c_{1}=0.$ Therefore $\dim (\fk_{0}\oplus\fa\oplus\fn)_{x}=\frac{p(p-1)+q(q+1)}{2}-\lbrace(p-q)+(q-1)+(q-1)+1\rbrace,$ which implies that $\dim {\bf K}_{0}{\bf A}{\bf N}(x)=\dim (\fk_{0}\oplus\fa\oplus\fn) - \dim(\fk_{0}\oplus\fa\oplus\fn)_x=p+q-1.$ 
	
	Now let  $1<k< q$. Then for $1\leq j\leq p-q$, if $p\neq q$, we have
	$$ A_{j,p-k+1}=\frac{1}{x^{p-k+1}+x^{p+k}}\lbrace -\sum_{i=p-q+1,i\neq p-k+1}^{p} (x^{i}+x^{2p-i+1}) A_{ji}+\sum_{i=1}^{j-1} x^{i}K_{ij}-\sum_{i=j+1}^{p-q} x^{i}K_{ji}\rbrace, $$
	for $p-q+1\leq j\leq p-k$ we have
	$$ A_{j,p-k+1}=\frac{1}{x^{p-k+1}+x^{p+k}}\lbrace -\sum_{\stackrel{i=j+1}{ i\neq p-k+1}}^{p} (x^{i}+x^{2p-i+1}) A_{ji} + \sum_{i=1}^{j-1}\lbrace x^{i}A_{ij}+x^{2p-i+1}D_{p-j+1,p-i+1}\rbrace -x^{2p-j+1}c_{p-j+1}\rbrace, $$
	and
	$$ D_{k,p-j+1}=\frac{1}{x^{p-k+1}+x^{p+k}}\lbrace-\sum_{\stackrel{i=j+1}{ i\neq p-k+1}}^{p}(x^{i}+x^{2p-i+1}) D_{p-i+1,p-j+1} + \sum_{i=1}^{j-1}\lbrace x^{i}A_{ij}+x^{2p-i+1}D_{p-j+1,p-i+1}\rbrace +x^{j}c_{p-j+1}\rbrace ,$$ 
	and for $p-k+2\leq j\leq p$, if $x^{p-k+1}\neq 0$ we have
	$$ A_{p-k+1,j}=\frac{1}{x^{p-k+1}}\lbrace - \sum_{i=1, i\neq p-k+1}^{j-1}\lbrace x^{i}A_{ij}+x^{2p-i+1}D_{p-j+1,p-i+1}\rbrace+x^{2p-j+1}c_{p-j+1}\rbrace, $$
	else we have
	$$ D_{p-j+1,k}=\frac{1}{x^{p+k}}\lbrace- \sum_{i=1,i\neq p-k+1}^{j-1}\lbrace x^{i}A_{ij}+x^{2p-i+1}D_{p-j+1,p-i+1}\rbrace -x^{p+k}c_{p-j+1}\rbrace.$$
	Summing the sides of the two equations including $c_{k}$ in (\ref{KAN-equ}) and using the fact that $x^{p-k+1}+x^{p+k}\neq0,$ one gets that $c_{k}=0.$
	These imply that $A_{j,p-k+1}$ and $ D_{k,p-j+l},$ where $1\leq j\leq p-k$ and $p-q+1\leq j\leq p-k$ and one of $A_{p-k+1,j}$ or $D_{p-j+1,k},$ where $p-k+2\leq j\leq p$ are linear functions of the other entiers and $c_{k}=0$. Therefore
	 $\dim (\fk_{0}\oplus\fa\oplus\fn)_{x}=\frac{p(p-1)+q(q+1)}{2}-\lbrace(p-k)+(q-k)+(k-1)+1\rbrace,$ which implies that $\dim {\bf K}_{0}{\bf AN}(x)=\dim (\fk_{0}\oplus\fa\oplus\fn) - \dim(\fk_{0}\oplus\fa\oplus\fn)_x=p+q-k.$

	Finally, let $k=q$. Then for $1\leq j\leq p-q$, if $p\neq q,$ we have
	$$ A_{j,p-q+1}=\frac{1}{x^{p-q+1}+x^{p+q}}\lbrace -\sum_{i=p-q+2}^{p} (x^{i}+x^{2p-i+1}) A_{ji}+\sum_{i=1}^{j-1} x^{i}K_{ij}-\sum_{i=j+1}^{p-q} x^{i}K_{ji}\rbrace, $$
	
	and for $p-q+2\leq j\leq p$, if $x^{p-q+1}\neq 0$ we have
	$$ A_{p-q+1,j}=\frac{1}{x^{p-q+1}}\lbrace -\sum_{i=1, i\neq p-q+1}^{j-1}\lbrace x^{i}A_{ij}+x^{2p-i+1}D_{p-j+1,p-i+1}\rbrace +x^{2p-j+1}c_{p-j+1} \rbrace, $$
	else we have
	$$ D_{p-j+1,q}=\frac{1}{x^{p+q}}\lbrace -\sum_{i=1,i\neq p-q+1}^{j-1}\lbrace x^{i}A_{ij}+x^{2p-i+1}D_{p-j+1,p-i+1}\rbrace - x^{j}c_{p-j+1} \rbrace ,$$
	These imply that $A_{j,p-q+1}$ where $1\leq j\leq p-q$ and one of the $A_{p-q+1,j}$ or $D_{p-j+1,q},$ where $p-q+2\leq j\leq p$ are linear functions of the other entries. Summing the sides of the two equations including $c_{q}$ in (\ref{KAN-equ}) and using the fact that  $x^{p-q+1}+x^{p+q}\neq0,$ one gets that $c_{q}=0.$ Therefore $\dim (\fk_{0}\oplus\fa\oplus\fn)_{x}=\frac{p(p-1)+q(q+1)}{2}-\lbrace(p-q)+(q-1)+1\rbrace,$ which implies that $\dim {\bf K}_{0}{\bf A}{\bf N}(x)=\dim (\fk_{0}\oplus\fa\oplus\fn) - \dim(\fk_{0}\oplus\fa\oplus\fn)_x=p.$
	
	{\bf Case 2:} $x\in \bigcap_{i=1}^q\Pi_i\diagdown \bigcap_{j=1}^{p-q}\mathscr{P}_j$. 
	
	{\it Claim 2:}  $\dim {\bf K}_{0}{\bf A}{\bf N}(x)=p-1$.
	
	{\it Proof of Claim 2:} Since $x\in \bigcap_{i=1}^q\Pi_i\diagdown \bigcap_{j=1}^{p-q}\mathscr{P}_j$, so $x^{p-i+1}+x^{p+i}=0$ for all $1\leq i\leq q$ and $x^{j_{0}}\neq 0 $  for some $1\leq j_{0}\leq p-q.$ Then the system of equations (\ref{KAN-equ}) reduces to the following system.
	
		\begin{eqnarray}\label{KAN-equ2}
	\begin{cases}
	\sum_{i=j+1}^{p-q}K_{ji}x^{i}-\sum_{i=1}^{j-1}K_{ij}x^{i}=0 \ , \quad 1\leqslant j\leqslant p-q,\quad  {\rm if}\ p\neq q,\\
	-\sum_{i=1}^{p-1}A_{i,p} x^{i} +c_{1}x^{p+1}-\sum_{i=p-q+1}^{p-1}D_{1,p-i+1 }x^{2p-i+1 }=0\\
	\sum_{i=1}^{p-1}A_{i,p} x^{i} +c_{1}x^{p}+\sum_{i=p-q+1}^{p-1}D_{1,p-i+1 }x^{2p-i+1 }=0\\
	\sum_{i=1}^{p-2}A_{i,p-1} x^{i} +c_{2}x^{p-1}+\sum_{i=p-q+1}^{p-2}D_{2,p-i+1}x^{2p-i+1 }=0\\
	\vdots\\
	\sum_{i=1}^{p-q+2}A_{i,p-q+3} x^{i} +c_{q-2}x^{p-q+3} +\sum_{i=p-q+1}^{p-q+2}D_{q-2,p-i+1}x^{2p-i+1}=0\\
	\sum_{i=1}^{p-q+1}A_{i,p-q+2} x^{i} +c_{q-1}x^{p-q+2}+D_{q-1,q}x^{p+q}=0\\
	\sum_{i=1}^{p-q}A_{i,p-q+1} x^{i} +c_{q}x^{p-q+1}=0 \\
	\end{cases}
	\end{eqnarray}
	
	If $p\neq q,$ then for $1\leq j\leq p-q$ and $j\neq j_{0},$ we have
	$$ K_{j_{0}j}=\frac{1}{x^{j_{0}}}\lbrace -\sum_{i=1, i\neq j_{0}}^{j-1} x^{i}K_{ij}+\sum_{i=j+1}^{p-q} x^{i}K_{ji}\rbrace. $$
	
	Hence for $p-q+1\leq j\leq p$ we have 
	$$ A_{j_{0}j}=\frac{1}{x^{j_{0}}}\lbrace -\sum_{i=1, i\neq j_{0}}^{j-1} x^{i} A_{ij} - \sum_{i=p-q+1}^{j-1}x^{2p-i+1}D_{p-j-1,p-i+1}+ x^{j}c_{p-j+1}\rbrace. $$
	
	These imply that $K_{j_{0}j}$ where $1\leq j\leq p-q,$ with $j\neq j_{0}$ and $ A_{j_{0}j}$ are linear functions of the other entiers, where $ p-q+1\leq j\leq p.$ Therefore $\dim (\fk_{0}\oplus\fa\oplus\fn)_{x}=\frac{p(p-1)+q(q+1)}{2}-((p-q-1)+q),$ which implies that $\dim {\bf K}_{0}{\bf A}{\bf N}(x)=\dim (\fk_{0}\oplus\fa\oplus\fn) - \dim(\fk_{0}\oplus\fa\oplus\fn)_x=p-1.$

	{\bf Case 3 :}  $x\in \bigcap_{i=1}^q\Pi_i\cap \bigcap_{j=1}^{p-q}\mathscr{P}_j$.
	
	{\it Claim 3:} In this case $\dim {\bf K}_{0}{\bf A}{\bf N}(x)=l$, where $ l=max\{j: 1\leq j\leq q \quad\textit{and}\quad x^{p+j}=-x^{p-j+1}\neq 0 \}.$
	
	{\it Proof of Claim 3:} Let $x^{p+l}=-x^{p-l+1}\neq 0$ where $ l=max\{j: 1\leq j\leq q \quad\textit{and}\quad x^{p+j}=-x^{p-j+1}\neq 0 \}.$ If $l=1$, then $x=r(e_p-e_{p+1})$ for some $r\in \R$, which implies that ${\bf K}_{0}{\bf A}{\bf N}(x)=\R x$ and so $\dim {\bf K}_{0}{\bf A}{\bf N}(x)=1$. If $l>1$, then the system of equations (\ref{KAN-equ}) reduces to the following system.
	\begin{eqnarray}\label{KAN-equ3}
	\begin{cases}
	\sum_{i=p-l+1}^{p-1}\lbrace A_{i,p} x^{i} + D_{1,p-i+1 }x^{2p-i+1 }\rbrace+c_{1}x^{p}=0\\
	\vdots \\
	\sum_{i=p-l+1}^{p-l+1}\lbrace A_{i,p-l+2} x^{i} + D_{l-1,p-i+1}x^{2p-i+1}\rbrace+c_{l-1}x^{p-l+2}=0.\\
	c_{l}x^{p-l+1}=0\\
	c_{i}x^{p-i+1}=0,\quad l+1\leq i\leq q.
	\end{cases}
	\end{eqnarray}
For $l+1\leq i\leq q,$ we have $x^{p-i+1}=0$ then $c_{i}\in \R.$
	Since $x^{p-l+1}\neq 0,$ we have $c_{l}=0$ and for $1\leq m\leq l-1$ we have
	$$ D_{ml}=\frac{1}{x^{p+l}}\lbrace -\sum_{i=p-q+1}^{p-m} x^{i} A_{i,p-m+1} - \sum_{i=m+1,i\neq l}^{q} x^{p+i}D_{mi}-x^{p-m+1}c_{m}\rbrace ,$$
	These imply that $ D_{ml}$ are linear functions of the other entiers, where $ 1\leq m\leq l-1$ and $c_{l}=0.$ Therefore $\dim (\fk_{0}\oplus\fa\oplus\fn)_{x}=\frac{p(p-1)+q(q+1)}{2}-l,$ and so $\dim {\bf K}_{0}{\bf A}{\bf N}(x)=\dim (\fk_{0}\oplus\fa\oplus\fn) - \dim(\fk_{0}\oplus\fa\oplus\fn)_x=l.$
	
\end{proof}
Let $G={\bf A}{\bf N}$, $x\in \R^{p,q}$ and $m=\dim G(x)$.  The following result states that either $0\leqslant m\leqslant q$ or $p\leqslant m\leqslant p+q-1$.
\begin{pro}\label{AN}
	 Let $G={\bf AN}$, which acts on $\R^{p,q}$ naturally. Let $x=\sum_{i=1}^{p+q}x^{i}e_{i} $ be a typical nonzero element of $\R^{p,q}$ and $\dim G(x)=m$. Then one of the following statements hold.
	 
	a) If $x\notin \bigcap_{i=1}^q\Pi_i$ then $m=p+q-k$, where $ k= min\{i: 1\leq i\leq q\ \ and\ \  x\notin \Pi_i\} $. 
	
	b) If $x\in \bigcap_{i=1}^q\Pi_i\diagdown \bigcap_{j=1}^{p-q}\mathscr{P}_j$, then $m=q$.
	
	c) If $x\in \bigcap_{i=1}^q\Pi_i\cap \bigcap_{j=1}^{p-q}\mathscr{P}_j$ then $m=max\{j: 1\leq j\leq q \quad\textit{and}\quad x^{p+j}=-x^{p-j+1}\neq 0 \}.$
\end{pro}
\begin{proof}
The proof is similar to that of Proposition \ref{KAN}, just the first line in the system of equation (\ref{KAN-equ}) is replaced by 
$$\sum_{i=p-q+1}^{p}A_{ji}(x^{i}+x^{2p-i+1})=0 \ , \quad 1\leqslant j\leqslant p-q,\quad  {\rm if}\ p\neq q.$$\end{proof}

Let $G={\bf K}'{\bf AN}$, where ${\bf K}'\subseteq {\bf K}_0$. We determine the orbits of the action of $G$ on the hyperquadrics in $\R^{p,q}$ by using Propositions \ref{KAN} and \ref{AN}. The hyperplane $\Pi_1$ divides $\R^{p,q}$ into two connected open sets $U_1$ and $U_2$ in $\R^{p,q}$, defined by $x^p>-x^{p+1}$ and $x^p<-x^{p+1}$ respectively. Then the sets $\Lambda^{p+q-1}\cap U_1$ and $\Lambda^{p+q-1}\cap U_2$ are nonempty connected open sets in $\Lambda^{p+q-1}$. Let $i\in\{1,2\}$. For any $x\in \Lambda^{p+q-1}\cap U_i$, the dimension of the orbit $G(x)$ is $p+q-1$ by Propositions \ref{KAN}-(a) and \ref{AN}-(a). Hence $G(x)$ is a connected open submanifold of $\Lambda^{p+q-1}\cap U_i$, for every $x\in \Lambda^{p+q-1}\cap U_i$. This implies that $G(x)=\Lambda^{p+q-1}\cap U_i$. Thus there are exactly two $p+q-1$ dimensional orbits in $\Lambda^{p+q-1}$. \\
The subspace $\Pi_1\diagdown\Pi_1\cap\Pi_2$ is the union of two connected open subsets $V_1$ and $V_2$ in $\Pi_1$, defined by $x^{p-1}>-x^{p+2}$ and $x^{p-1}<-x^{p+2}$. Let $i\in\{1,2\}$ and $x\in V_i$. Then $\dim G(x)=p+q-2$ by Propositions \ref{KAN}-(a) and \ref{AN}-(a). Using the fact that $\Lambda^{p+q-1}\cap V_i$ is a $p+q-2$ dimensional connected manifold, one gets that $G(x)= \Lambda^{p+q-1}\cap V_i$. Therefore there are only two orbits with dimension $p+q-2$ in $\Lambda^{p+q-1}$. 
And so, for any $2\leqslant j\leqslant q$, $\bigcap_{i=1}^{j-1}\Pi_i\diagdown\bigcap_{i=1}^j\Pi_i$ 	is the union of two nonempty open sets $W_1$ and $W_2$ in $\bigcap_{i=1}^{j-1}\Pi_i$, defined by $x^{p-j+1}>-x^{p+j}$ and $x^{p-j+1}<-x^{p+j}$ respectively. For any $x\in W_i$, where $i\in \{1,2\}$, we have $\dim G(x)=p+q-j$ by Propositions \ref{KAN}-(a) and \ref{AN}-(a). The fact that $\Lambda^{p+q-1}\cap W_i$ is a $p+q-j$ dimensional connected manifold implies that $G(x)=\Lambda^{p+q-1}\cap W_i$, i.e. there are exactly two orbits with dimension $p+q-j$ in $\Lambda^{p+q-1}$. Thus the only remaining set of the points of $\Lambda^{p+q-1}$ to study their orbits is  $\bigcap_{i=1}^q\Pi_i\cap \Lambda^{p+q-1}$.\\
We have 
$$\bigcap_{i=1}^q\Pi_i\cap \Lambda^{p+q-1}=\R w_1\oplus \cdots \oplus \R w_q.$$
Let $k\in \{1,..., q\}$. Then $w_k$ is fixed by ${\bf K}_0$, and so by ${\bf K}'$. Also ${\bf A}(w_k)=\R_+w_k$ and ${\bf N}(w_k)=\R w_1\oplus \cdots\oplus \R w_{k-1} +w_k$ by Remark \ref{rem1}. This shows that for any point $x=\Sigma_{j=1}^{l}r_jw_j$, where $r_l\neq 0$, $$G(x)=\R w_1\oplus \cdots\oplus \R w_{l-1}+\R_+(r_lw_l).$$
Hence there are two $l$-dimensional orbits depending on the sign of $r_l$. Thus we proved the following corollary.

\begin{cor}\label{co-Lambda}
Let $G={\bf K}'{\bf A}{\bf N}$, where ${\bf K}'\subseteq {\bf K}_0$, $x\in \Lambda^{p+q-1}$ and $m=\dim G(x)$. Then $m\in \{1,..., q\}\cup \{p,..., p+q-1\}$ and there are exactly two orbits with dimension $m$ in $\Lambda^{p+q-1}$. Furthermore,

(a) if $m\in \{1,..., q\}$ then $G(x)$ is either $\Sigma_{i=1}^{m-1}\R w_{i}+\R_+ w_m$ or $\Sigma_{i=1}^{m-1}\R w_{i}+\R_- w_m$.

(b) if $m\in \{p,..., p+q-1\}$, then $G(x)$ is one of the connected components of  $\bigcap_{i=0}^{p+q-m-1}\Pi_i\diagdown\bigcap_{i=0}^{p+q-m}\Pi_i$  (here, $\Pi_0$ denotes $\R^{p,q}$).

(c) $G(x)$ is not dependent on the choice of ${\bf K}'$.  
\end{cor}
%
%
%
Let $k\in \{1,...,q\}$. By a similar argument of that of the proof of Corollary \ref{co-Lambda}
one gets that the subspace $\bigcap^{k-1}_{i=0}\Pi_i\diagdown\bigcap_{i=0}^k\Pi_i$ is the union of two disjoint connected open sets $U_1$ and $U_2$ in $\bigcap_{i=0}^{k-1}\Pi_i$ (here $\Pi_0=\R^{p,q}$). Let $j\in \{1,2\}$. Then $\mathbb{H}^{p,q-1}(r)\cap U_j$ is a $p+q-k$ dimensional connected submanifold of $\mathbb{H}^{p,q-1}(r)$ which is preserved by $G$ by Remark \ref{rem1}, where $G={\bf K}'{\bf A}{\bf N}$. By Propositions \ref{AN}-(a) and \ref{KAN}-(a), we have $\dim G(x)=p+q-k$ for any $x\in \mathbb{H}^{p,q-1}(r)\cap U_j$. The orbit $G(x)$ is connected, so $G(x)=H^{p+q-k}(r)\cap U_j$. This shows that there are only two  orbits with dimension $p+q-k$ in $\mathbb{H}^{p,q-1}(r)$. Therefore the following corollary is proved.

\begin{cor}\label{co-H^{p,q-1}}
Let $G={\bf K}'{\bf A}{\bf N}$, where ${\bf K}'\subseteq {\bf K}_0$, $x\in \mathbb{H}^{p,q-1}(r)$, where $r\in \R_{+}$ and $m=\dim G(x)$. Then $m\in \{p,p+1,..., p+q-1\}$ and there are exactly two orbits with dimension $m$ in $\mathbb{H}^{p,q-1}(r)$. Furthermore,

(a) if $m\in \{p, p+1, ..., p+q-1\}$ then $G(x)$ is one of the connected components of  $\mathbb{H}^{p,q-1}(r)\cap (\bigcap_{i=0}^{p+q-m-1}\Pi_i\diagdown\bigcap_{i=0}^{p+q-m}\Pi_i)$.

(b) $G(x)$ is not dependent on the choice of ${\bf K}'$. 
\end{cor}

Now we are going to study the orbits of the action of $G$ on $\mathbb{S}^{p-1,q}$, where $G={\bf K}'{\bf A}{\bf N}$ and ${\bf K}'\subseteq {\bf K}_0$. Let $x\in \mathbb{S}^{p-1,q}(r)$. By a similar argument of that of the proof of Corollary \ref{co-Lambda} one gets that for any $x\in \mathbb{S}^{p-1,q}(r)\cap (\bigcap_{i=0}^{p+q-m-1}\Pi_i\diagdown\bigcap_{i=0}^{p+q-m}\Pi_i)$, where $p\leqslant m\leqslant p+q-1$, the orbit $G(x)$ is the connected component of  $\mathbb{S}^{p-1,q}(r)\cap (\bigcap_{i=0}^{p+q-m-1}\Pi_i\diagdown\bigcap_{i=0}^{p+q-m}\Pi_i)$ containing $x$. This shows that for each $m\in \{p,p+1,..., p+q-1\}$ there are exactly two orbits with dimension $m$ in $\mathbb{S}^{p-1,q}(r)$. Furthermore, these orbits are not dependent on the choice of ${\bf K}'$ by Propositions \ref{KAN}-(a) and \ref{AN}-(a). Thus the only remaining points of $\mathbb{S}^{p-1,q}(r)$ to study their orbits is  $\bigcap_{i=1}^q\Pi_i\cap \mathbb{S}^{p-1,q}(r)$. If $p=q$, this intersection is empty.\\
Let $\mathbb{W}^p=\bigcap^q_{i=1}\Pi_i$. If $p>q$ then the intersection $\mathbb{W}^{p}\cap \mathbb{S}^{p-1,q}(r) $ is equal to the cylinder $C^{p-1}_q(r)$ defined by
$\mathbb{W}^{p}\cap \mathbb{S}^{p-1,q}(r)=\mathbb{S}^{p-q-1}(r)\times \R^{q}$, where $\mathbb{S}^{p-q-1}(r)$ is the $p-q-1$ dimensional Euclidean sphere with radius $r$ in $\R e_{1}\oplus \cdots\oplus \R e_{p-q}\subset \mathbb{W}^p$ and $\R^q=\bigcap_{i=1}^{p-q}\mathscr{P}_i\cap \bigcap_{j=1}^q \Pi_{j}$.  If $p=q+1$, then $C^{p-1}_q(r)=\{\pm r\}\times \R^q$. If $p>q+1$, then the cylinder  $C^{p-1}_q(r)$ is connected, and if $q>1$ then the complement $\mathbb{S}^{p-1,q}(r)\diagdown C^{p-1}_q(r) $ is connected too. For $q=1$ the complement  $\mathbb{S}^{p-1,1}(r)\diagdown C^{p-1}_1(r)$  has two connected components. 

Let $p\geqslant q+1$. The action of $ {\bf K}_{0}$ and ${\bf A}$ on a point $x+\sum_{j=1}^{q}r^{j}w_{j}\in C^{p-1}_q(r) $ is given by
$$x+\sum_{j=1}^{q}r^{j}w_{j}\longmapsto Ax+\sum_{j=1}^{q}r^{j}w_{j}\quad {\rm and}\quad  x+\sum_{j=1}^{q}r^{j}w_{j}\longmapsto x+\sum_{j=1}^{q}e^{-c_{j}}r^{j}w_{j},$$
with $x\in \mathbb{S}^{p-q-1}(r)\subset \R^{p-q}$ and $r^{j}\in \R$, where $A\in SO(p-q)$, $c_{j}\in \R$, respectively. By Remark \ref{rem1} we have  ${\bf N}(x+\sum_{j=1}^{q}r^{j}w_{j})\subseteq x+\bigoplus_{j=1}^{q}\R w_j$. On the other hand, by the proof of Theorem \ref{N}-case(2), ${\bf N}(x+\sum_{j=1}^{q}r^{j}w_{j})$ is a $q$-dimensional orbit, which implies that ${\bf N}(x+\sum_{j=1}^{q}r^{j}w_{j})=x+\bigoplus_{j=1}^{q}\R w_j$.
It follows that ${\bf K}_{0}{\bf A}{\bf N}$ leaves the cylinder $C^{p-1}_q(r)$ invariant. More precisely, ${\bf K}_{0}= SO(p-q)$ acts canonically on $\mathbb{S}^{p-q-1}(r)$ and trivially on $\bigoplus_{j=1}^{q} \R w_{j}$, ${\bf A}$ and ${\bf N}$ act trivially on $\mathbb{S}^{p-q-1}(r)$, ${\bf N}$ acts
transitively on $\bigoplus_{j=1}^{q}\R w_{j}$ (here $\bigoplus_{j=1}^{q}\R w_{j}$ is embedded in $C^{p-1}_q(r)$), and ${\bf A}$ has $2^{2q-1}$ orbits on $\bigoplus_{j=1}^{q}\R w_{j}$ (namely, one orbit of dimension zero, $2q\choose 1$ orbits of dimension one, ...,  $2q\choose q$ orbits of dimension $q$). This shows that the orbits of $AN$ on $C^{p-1}_q(r)$ are precisely the $p$-dimensional degenerate affine subspaces $x+\bigoplus_{j=1}^{q}\R w_{j}$ with $x\in \mathbb{S}^{p-q-1}(r)\subset C^{p-1}_q(r)$. If $p=q+1$, then $K_0$ is trivial and so the action of $G$ reduces to the action of ${\bf A}{\bf N}$, which implies that the orbits are $\{\pm r\}\times \R^q$ in the cylinder.

Let $p>q+1$. Since ${\bf K}_{0}$ acts transitively on $\mathbb{S}^{p-q-1}(r)$ we see that the subgroup $Q ={\bf K}_{0}{\bf AN}$ acts transitively on $C^{p-1}_q(r)$. If ${\bf K}^{\prime}$ is a subgroup of ${\bf K}_0$, then the orbits of ${\bf K}^{\prime}{\bf AN}$ on the cylinder $C^{p-1}_q(r)$ correspond bijectively to the orbits of ${\bf K}^{\prime}$ on the sphere $\mathbb{S}^{p-q-1}(r)$. Now we can conclude the following corollary.

\begin{cor}\label{co-S^{p-1,q}}
	Let $G={\bf K}'{\bf AN}$, where ${\bf K}'\subseteq {\bf K}_0$, $x\in \mathbb{S}^{p-1,q}(r)$, where $r\in \R_{+}$, and $m=\dim G(x)$. Then $m\in \{q,... , p+q-1\}$ and the following statements hold.

(a) If $m\in \{p,...,p+q-1\}$, then there are exactly two orbits with dimension $m$ in $\mathbb{S}^{p-1,q}(r)$ determined as the connected components of  $\mathbb{S}^{p-1,q}(r)\cap(\bigcap_{i=0}^{p+q-m-1}\Pi_i\diagdown\bigcap_{i=0}^{p+q-m}\Pi_i)$. Furthermore, these orbits are not dependent on the choice of ${\bf K}'$.

(b) If $m\in \{q,...,p-1\}$, then $x=y+\Sigma_{j=1}^qr^j w_j \in \mathbb{S}^{p-1,q}(r)\cap \mathbb{W}^p= C^{p-1}_q(r)$, and $$G(x)={\bf K}'(y)\times \sum_{j=1}^q\R w_j.$$
\end{cor} 

The following results generalizes Theorem 4.2 and Corollary 4.3 of \cite{BDV}.

\begin{theorem}\label{mainresult}
	Let $G={\bf K}'{\bf AN}$, where ${\bf K}'\subseteq {\bf K}_0$, $x\in \R^{p,q}$ a nonzero point and $m=\dim G(x)$.
	
	(a) If $x\in \Lambda^{p+q-1}$ then $m\in \{1,..., q\}\cup \{p,..., p+q-1\}$ and there are exactly two orbits with dimension $m$ in $\Lambda^{p+q-1}$. If  $m\in \{1,..., q\}$ then $G(x)$ is either $\Sigma_{i=1}^{m-1}\R w_{i}+\R_+ w_m$ or $\Sigma_{i=1}^{m-1}\R w_{i}+\R_- w_m$. If $m\in \{p,..., p+q-1\}$, then $G(x)$ is one of the connected components of $\bigcap_{i=0}^{p+q-m-1}\Pi_i\diagdown\bigcap_{i=0}^{p+q-m}\Pi_i$. Furthermore, $G(x)$ is not dependent on the choice of ${\bf K}'$.
	
	(b) If $x\in \mathbb{H}^{p,q-1}(r)$ for some $r>0$, then	
	$m\in \{p,p+1,..., p+q-1\}$ and there are exactly two orbits with dimension $m$ in $\mathbb{H}^{p,q-1}(r)$. If $m\in \{p, p+1, ..., p+q-1\}$ then $G(x)$ is one of the connected components of  $\mathbb{H}^{p,q-1}(r)\cap (\bigcap_{i=0}^{p+q-m-1}\Pi_i \diagdown\bigcap_{i=0}^{p+q-m}\Pi_i)$. Furthermore, $G(x)$ is not dependent on the choice of ${\bf K}'$. 
	
	(c) If $x\in \mathbb{S}^{p-1,q}(r)$, for some $r>0$, then $m\in \{q,.., p+q-1\}$ and the following statements are hold.
	
	($c_1$) If $m\in \{p,...,p+q-1\}$, then there are exactly two orbits with dimension $m$ in $\mathbb{S}^{p-1,q}(r)$ determined as the connected components of $\mathbb{S}^{p-1,q}(r)\cap(\bigcap_{i=0}^{p+q-m-1}\Pi_i\diagdown\bigcap_{i=0}^{p+q-m}\Pi_i)$. Furthermore, these orbits are not dependent on the choice of ${\bf K}'$.
	
	($c_2$) If $m\in \{q,...,p-1\}$, then $x=y+\Sigma_{j=1}^qr^j w_j \in \mathbb{S}^{p-1,q}(r)\cap \mathbb{W}^p= C^{p-1}_q(r)$, and $$G(x)={\bf K}'(y)\times \sum_{j=1}^q\R w_j.$$
\end{theorem}

\begin{proof}
	It is an immediate consequence of Corollaries \ref{co-Lambda}, \ref{co-H^{p,q-1}} and \ref{co-S^{p-1,q}}.
\end{proof}

Theorem \ref{mainresult} shows that the orbits of ${\bf K}^{\prime}{\bf AN}$ on $\R^{p,q}\diagdown \mathbb{W}^{p}$ are independent of the choice of ${\bf K}^{\prime}$. Thus we get the following remarkable consequence.

\begin{pro}\label{degspace}
	There exist cohomogeneity one actions on $\R^{p,q}$, $p>q+1\geqslant 2$, which are
	orbit-equivalent on the complement of an $p$-dimensional degenerate subspace $\mathbb{W}^{p}$ of $\R^{p,q}$ and not orbit-equivalent on $\mathbb{W}^{p}.$
\end{pro}

\bibliographystyle{amsplain}

\end{document}